\documentclass[12pt]{article}
\usepackage{amssymb, amsthm, amsmath, fixmath}

\newtheorem{theorem}{Theorem}[section]
\newtheorem{lemma}[theorem]{Lemma}
\newtheorem{corollary}[theorem]{Corollary}
\newtheorem{assumption}[theorem]{Assumption}

\newcommand{\RR}{\mathbb R}
\newcommand{\CC}{\mathbb C}

\begin{document}
\title{Distinct distances on algebraic curves in the plane}

\author{J\'anos Pach\thanks{EPFL, Lausanne and R\'enyi Institute,
    Budapest. Supported by NSF Grant CCF-08-30272, by OTKA under EUROGIGA
    projects GraDR and ComPoSe 10-EuroGIGA-OP-003, and by Swiss National
    Science Foundation Grants 200020-144531 and 20021-137574. 
    Email: {\tt pach@cims.nyu.edu}}
  \and 
    Frank de Zeeuw\thanks{EPFL, Lausanne. Supported by 
    Swiss National Science Foundation Grants 200020-144531 and 200021-137574.
    Email: {\tt fdezeeuw@gmail.com}}}

\date{}  
\maketitle

\begin{abstract}
Let $S$ be a set of $n$ points in $\RR^2$ contained in an algebraic curve $C$ of degree $d$.
We prove that the number of distinct distances determined by $S$ is at least $c_dn^{4/3}$, 
unless $C$ contains a line or a circle.

We also prove the lower bound $c_d' \min\{m^{2/3}n^{2/3}, m^2, n^2\}$
for the number of distinct distances between $m$ points on one irreducible plane algebraic curve and $n$ points on another,
unless the curves are parallel lines, orthogonal lines, or concentric circles.
This generalizes a result on distances between lines of Sharir, Sheffer, and Solymosi \cite{sharir:distances}.

\end{abstract}

\bigskip

\section{Introduction}

A famous conjecture of Erd\H os, first mentioned in \cite{erdos:distances}, 
states that any set of $n$ points in $\RR^2$
determines at least $\upOmega(n/\sqrt{\log n})$ distinct distances.
Over the years this has been a central problem in combinatorial geometry,
with many successive improvements of the best-known lower bound (see \cite{brass:problems}, Section 5.3).
In \cite{guth:distances}, Guth and Katz established an almost complete solution,
proving the lower bound $\upOmega(n/\log n)$.
A new element in their proof was the use of tools from algebraic geometry.

A related problem posed by Purdy (see \cite{brass:problems}, Section 5.5) is to determine
the least number of distances that can occur between two collinear point sets,
say $n$ points on a line $l_1$ and $n$ points on a line $l_2$.
If $l_1$ and $l_2$ are parallel or orthogonal, 
then $O(n)$ distances are possible, 
but otherwise there should be substantially more.
This was proved by Elekes and R\'onyai in \cite{elekes:polynomials}, 
where they derived it from a more general result about polynomials,
which they proved using a combination of combinatorial and algebraic methods.
In \cite{elekes:note}, Elekes specialized these methods to Purdy's question,
resulting in a lower bound of $\upOmega(n^{5/4})$ on the number of distances, 
if the two lines are not parallel or orthogonal.
Recently, Sharir, Sheffer, and Solymosi improved this bound to $\upOmega(n^{4/3})$ in \cite{sharir:distances},
again using algebraic methods. 
In \cite{schwartz:extensions}, Schwartz, Solymosi, and De Zeeuw extended the general result of Elekes
and R\'onyai in several ways, one of which resulted in an unbalanced version
of Purdy's problem, where one line contains $m$ points and the other $n$.
This was also strengthened for Purdy's problem in \cite{sharir:distances}, 
to a lower bound $\upOmega(\min\{m^{2/3}n^{2/3}, m^2, n^2\})$.

The aim of this paper is to extend the result of \cite{sharir:distances}
from lines to arbitrary plane algebraic curves (see Subsection \ref{subsec:alggeom} for definitions).
The results take several forms; perhaps the most interesting of them is the following.

\begin{theorem}\label{thm:onecurve}
 Let $C$ be a plane algebraic curve of degree $d$ that does not contain a line or a circle.
 Then any set of $n$ points on $C$ determines at least $c_d n^{4/3}$ distinct distances,
 for some $c_d>0$ depending only on $d$. 
\end{theorem}
 
Note that if the curve is a line or a circle, 
$O(n)$ distances are possible for certain point sets,
including any sequence of equidistant points.
With the current proof, the constant $c_d$ roughly comes out to $cd^{-7}$ for an absolute constant $c$.
We have not made a serious effort to optimize it, 
but in a remark at the end of Section \ref{sec:proofoftheorems} we suggest some improvements.
 
Theorem \ref{thm:onecurve} is a direct consequence of the proof of the following Theorem.
  
\begin{theorem}\label{thm:twocurves}
 Let $C_1,C_2$ be two irreducible plane algebraic curves of degree at most $d$ which are not parallel lines, orthogonal lines, or concentric circles.
 Then for any $m$ points on $C_1$ and $n$ points on $C_2$, the number of distinct distances between the two sets is at least 
 $c_d'\cdot \min\{m^{2/3}n^{2/3}, m^2, n^2\}$,
 for some $c_d'>0$ depending only on $d$. 
\end{theorem}
In the excluded cases, $O(n)$ distances are again possible for certain point sets.
One can also deduce a result for two curves that are not necessarily irreducible,
but it would be more inconvenient to state:
The curves should not both contain a line, such that the two lines are parallel or orthogonal,
and they should not both contain a circle, such that the two circles are concentric.
 
While this paper was being finalized, a number of related results were made public.
In \cite{charalambides:gaps}, Charalambides establishes 
a version of Theorem \ref{thm:onecurve} with the weaker lower bound $c_d n^{5/4}$. 
He combines the technique of \cite{elekes:note} with analytic as well as algebraic tools,
and even extends it to higher dimensions, with a more complicated set of exceptions.
In \cite{sharir:threepoints}, Sharir and Solymosi show,
using a method based on that of \cite{sharir:distances},
that between three non-collinear points and $n$ other points
there are $\upOmega(n^{6/11})$ distinct distances.
In \cite{sheffer:few}, Sheffer, Zahl, and De Zeeuw
extend the method of \cite{sharir:distances} to the case 
where one set of points in $\RR^2$ is constrained to a line, 
while the other is unconstrained.
Finally, in \cite{raz:polynomials}, Raz, Sharir, and Solymosi use the approach of \cite{sharir:distances}
to improve the bounds in the more general problem about polynomials of \cite{elekes:polynomials}.

We say a few words about our proofs compared to those of the similar results mentioned above.
Both \cite{elekes:note} and \cite{sharir:distances} derive their bound 
by constructing a set of new curves and applying an incidence bound to them.
The way these curves are constructed relies heavily on the fact that lines can be parametrized.
This makes it possible to extend their methods to parametrizable curves,
but makes it harder to extend to general algebraic curves,
which are defined by an implicit equation.
In \cite{charalambides:gaps}, this is overcome using the Implicit Function Theorem, 
which allows implicit curves to be ``parametrized'' analytically.
One important new element of our proofs is that we construct the new curves 
in an implicit and algebraic way (see in particular \eqref{eq:defcij} in Section \ref{sec:proofoftheorems}),
making parametrization unnecessary.

To apply the incidence bound to the constructed curves,
one needs to show that the curves have small intersection, and in particular that they are distinct.
In \cite{elekes:note} and \cite{sharir:distances}, this is relatively easy because the curves have low degree.
In \cite{charalambides:gaps}, it is done using concepts from the theory of rigidity.
We observe instead that, if some of the constructed curves have large intersection,
this must be due to some kind of symmetry of the original curve.
The only curves that have too many symmetries are lines and circles,
which is why these are the exceptions in Theorem \ref{thm:onecurve}.
In Theorem \ref{thm:twocurves}, 
the exceptions are those pairs of curves that have too many symmetries \textit{in common}.

\vspace{10pt}

In Subsection \ref{subsec:alggeom}, we define algebraic curves and state several results from algebraic geometry,
in a way that we hope is accessible to readers that are not familiar with algebraic geometry.
In Subsection \ref{subsec:incidences} we introduce the incidence bound that is central to our proof,
and in Subsection \ref{subsec:symmetries} we prove two elementary results about symmetries of curves.
 In Section \ref{sec:proofoftheorems}, we give the proof of our two main theorems,
 up to the more delicate proof of one lemma, 
 which we give separately in Section \ref{sec:infinite}.


\section{Preliminaries}

\subsection{Definitions and tools from algebraic geometry}\label{subsec:alggeom}
We define a set $C\subset \RR^2$ to be a {\it (plane) algebraic curve} if there is a polynomial $f\in \RR[x,y]\backslash\{0\}$ such that 
$$C = Z_{\RR}(f) = \{ (a,b)\in\RR^2: f(a,b) = 0\}.$$
We define the {\it degree} of $C$ to be the degree of a minimum-degree polynomial $f$ such that $C=Z_{\RR}(f)$. 
If a curve has degree $2$ and is not the union of two lines, we call it a {\it conic}.
Note that in our definition, an algebraic curve can also be a {\it finite} set; take for instance $Z_{\RR}((x(x-1))^2 + y^2)$.
Fortunately, the size of this finite set is bounded in terms of the degree of the polynomial (by Theorem \ref{thm:milnorthom} below), so this is not a problem in our main theorem, since for finite sets of bounded size the statement is trivial.

We say that a plane algebraic curve $C$ is {\it irreducible} if we can write $C=Z_{\RR}(f)$ with a polynomial $f\in \RR[x,y]$ that is irreducible over $\RR$.
By an {\it irreducible component} of an algebraic curve $Z_{\RR}(f)$ we mean an irreducible algebraic curve $Z_{\RR}(h)$ for some nonconstant $h\in \RR[x,y]$ that divides $f$;
it then follows that $Z_{\RR}(h)\subset Z_{\RR}(f)$.
We say that two curves $Z_{\RR}(f)$ and $Z_{\RR}(g)$ {\it have a common component} if there is a nonconstant polynomial $h\in \RR[x,y]$ that divides $f$ and $g$; it then follows that $Z_{\RR}(h)\subset Z_{\RR}(f)\cap Z_{\RR}(g)$.


We frequently use B\'ezout's inequality in the plane, which is an upper bound on the number of intersection points of algebraic curves.
It is in fact an equality (B\'ezout's theorem) if one defines multiplicities of intersection points and works in the complex projective plane, but for us the inequality suffices.
See \cite[Lemma 14.4]{gibson:elementary} for exactly the statement below, 
or \cite[Theorem 8.7.10]{cox:ideals} for the complex version.

\begin{theorem}[B\'ezout's inequality]\label{thm:bezout}
 Two algebraic curves in $\RR^2$ with degrees $d_1$ and $d_2$ have at most $d_1\cdot d_2$ intersection points, unless they have a common component.
\end{theorem}

Although our objects of study are curves in the plane, our proofs involve curves in higher dimensions. Specifically, we encounter curves in $\RR^4$ that are zero sets of three polynomials. To analyze these real curves, we also consider complex zero sets.

Given polynomials $f_1,\ldots,f_k\in\RR[x_1,\dots,x_D]$, we define the (real) \emph{zero set}
\[Z_{\RR}(f_1,\ldots,f_k) = \{p\in\RR^D: \forall i ~f_i(p) = 0\},\]
and given polynomials $f_1,\ldots,f_k\in\CC[x_1,\dots,x_D]$, we define the (complex) zero set
\[Z_{\CC}(f_1,\ldots,f_k) = \{p\in\CC^D: \forall i ~f_i(p) = 0\}.\]
We also refer to zero sets as \emph{varieties}.
For a real zero set $X=Z_{\RR}(f_1,\ldots,f_k)$, we write $X^\CC$ to denote the complex zero set $Z_{\CC}(f_1,\ldots,f_k)$; similarly, for a complex zero set $X\subset \CC^D$ we define its \emph{real part} $X\cap \RR^D$.

For the definitions of the {\it dimension} and \emph{degree} of a complex zero set we refer to \cite{harris:algebraic}.
For real zero sets, these notions can be problematic, so we will avoid them by relying on the complex versions (except in $\RR^2$).
For a real zero set $Z_{\RR}(f_1,\ldots,f_k)$, we define its {\it complex dimension} to be the dimension of $Z_{\CC}(f_1,\ldots,f_k)$.
If a complex zero set has dimension one, then we call it a {\it complex (algebraic) curve}; note that a complex curve may contain zero-dimensional components (a complex zero set may have components of different dimension, and its dimension is the maximum).
If a real zero set has complex dimension one, we call it a {\it real (algebraic) curve}; note that a real curve may be a finite set.
In $\RR^2$ this definition coincides with the one above.

The degree of a complex zero set $Z_\CC(f_1,\ldots,f_k)$ is bounded above by the product of the degrees of the polynomials $f_i$, by the generalized B\'ezout's inequality (see for instance \cite{heintz:definability}).
By our definition, a complex curve $C$ in $\CC^2$ of degree $d$ can be written as $Z_\CC(f)\cup S$ for a single $f\in \CC[x,y]$ of degree at most $d$, and a finite set $S$ of size at most $d$.
The real part $C\cap \RR^2$ is then a real curve of degree at most $4d$:  $Z_\CC(f)\cap \RR^2$ is the zero set of $(\mathrm{Re}f)^2+(\mathrm{Im})^2$, where $\mathrm{Re}f$ and $\mathrm{Im}f$ are the real and imaginary part of $f$, which are real polynomials of degree at most $d$; $S\cap \RR^2$ is the zero set of at most $d$ polynomials of the form $(x-x_0)^2+(y-y_0)^2$.

A complex curve in $\CC^D$ is {\it irreducible} if it is not the union of two proper subsets which are curves; an {\it irreducible component} is a subset which is an irreducible curve; and two curves {\it have a common component} if there is a curve which is a subset of both.
The number of irreducible components of a curve is bounded above by its degree.

We also need a version of B\'ezout's inequality for $\RR^D$.
Over $\CC$, there are far-reaching generalizations of B\'ezout's inequality, like the one from \cite{heintz:definability} mentioned above.
But over $\RR$ some such generalizations may fail.
For instance, when a real zero set is finite, the size of the set need not be bounded above by the product of the degrees of the polynomials defining it: Take in $\RR^3$ the intersection of the plane $z=0$ with the zero set of the polynomial
$(x(x-1)(x-2))^2 + (y (y-1)(y-2))^2$,
which is a set of $9$ points, while the product of the degrees of the polynomials is $6$.

To overcome this complication, one could carefully use the complex version of B\'ezout's inequality, but we instead rely on a bound on the number of connected components of a real zero set.
The bound that we use is due to Oleinik-Petrovskii, Milnor, and Thom. 
For an exposition see \cite[Chapter 7]{basu:real}.
A {\it connected component} of an algebraic curve in $\RR^D$
is a connected component in the Euclidean topology on $\RR^D$.
Note that this is not the same as an irreducible component;
for instance, the curve $y^2 = x^3 -x$ in $\RR^2$ has one irreducible component,
but two connected components.

\begin{theorem}\label{thm:milnorthom}
 A zero set in $\RR^D$ defined by polynomials of degree at most $d$ has at most $(2d)^D$ connected components.
\end{theorem}

\subsection{Incidence bound}\label{subsec:incidences}
We will use an incidence bound from combinatorial geometry,
due to Pach and Sharir \cite{pach:repeated, pach:incidences};
the version in Theorem \ref{thm:pachsharir} below is proved in \cite{kaplan:simple}.

Let $P\subset\RR^D$ and let $\Gamma$ be a set of curves $\RR^D$.
We define $I(P,\Gamma)$ to be the set of {\it incidences},
i.e., the set of pairs $(p, \gamma)\in P\times \Gamma$ such that $p\in \gamma$.
We say that $P$ and $\Gamma$ form a {\em system with $k$ degrees of freedom}
if there is a {\em multiplicity} $M$ such that any two curves in $\Gamma$ intersect in at most $M$ points of $P$,
and any $k$ points of $P$ belong to at most $M$ curves in $\Gamma$.

\begin{theorem}[Pach-Sharir]\label{thm:pachsharir}
 If a set $P$ of points in $\RR^2$ and a set $\Gamma$ of algebraic curves in $\RR^2$ of degree at most $d$ form a system with two degrees of freedom and multiplicity $M$, then
 $$|I(P,\Gamma)| \leq C_{d,M} \cdot  \max\{|P|^{2/3}|\Gamma|^{2/3}, |P|, |\Gamma|\},$$
 where $C_{d,M}$ is a constant depending only on $d$ and $M$.
\end{theorem}


We will deduce a version for curves in higher dimensions, using a ``generic projection trick'' as used by Solymosi and Tao in \cite[Section 5.1]{solymosi:incidence}.

\begin{corollary}\label{cor:incidences}
  If a set $P$ of points in $\RR^D$ and a set $\Gamma$ of real algebraic curves in $\RR^D$, each defined by $e$ polynomials of degree at most $d$, form a system with two degrees of freedom and multiplicity $M$, then
 $$|I(P,\Gamma)| \leq C_{D,e,d,M} \cdot  \max\{|P|^{2/3}|\Gamma|^{2/3}, |P|, |\Gamma|\},$$
 where $C_{D,e,d,M}$ is a constant depending only on $D$, $e$, $d$, and $M$.
\end{corollary}
\begin{proof}
 We will use a generic projection from $\RR^D$ to $\RR^2$ and then apply Theorem \ref{thm:pachsharir}.
We have to face the technical complication that the projection of a real algebraic curve need not be a real algebraic curve; for instance, a circle may be projected down to a line segment.
To deal with this, we partly have to work with the complex versions of the curves.
We write $\Gamma^\CC$ for the set $\{\gamma^\CC: \gamma\in \Gamma\}$; these are complex curves of degree at most $d^e$ (here we use the fact that by our definition, the real curves in $\Gamma$ have complex dimension one).
We consider $P\subset \RR^D$ as a subset of $\CC^D$ in the natural way, and for any real linear transformation we consider the corresponding complex linear transformation (i.e., the one given by the same real matrix); we abuse notation slightly and use the same symbol for the real transformation and the complex transformation.

Let us make explicit what we mean by ``generic''.
We first apply a suitable linear transformation $\varphi$ to $\RR^D$, and then we apply the standard projection $\psi:\RR^D\to\RR^2$ defined by $(x_1,\ldots, x_D)\mapsto(x_1,x_2)$. 
 We claim that $\varphi$ can be chosen so that $\pi = \psi\circ\varphi$ has the two properties below.
 If this is true, then we will be able to apply Theorem \ref{thm:pachsharir} to the projected points and curves, which proves the required bound.
 
  \begin{enumerate}
  \item[(1)]\label{eq:bijective} The map $\pi$ is bijective on $P$ and
  induces a bijection from $I(P,\Gamma^\CC)$ to $I(\pi(P),\pi(\Gamma^\CC))$;
  \item[(2)] For distinct $\gamma,\gamma'\in\Gamma$, $\pi(\gamma^\CC)$ and $\pi(\gamma'^\CC)$ are distinct algebraic curves in $\CC^2$ of degree at most $d$.
  \end{enumerate} 
 
We show that the \emph{complex} linear transformations $\varphi$ for which $\pi$ fails to have these properties lie in a lower-dimensional subvariety of the $D^2$-dimensional space $M_D(\CC)$ of (matrices of) linear transformations of $\CC^D$.
 This will then imply that there is a \emph{real} matrix with these properties, because $M_D(\RR)$ cannot lie within a proper subvariety of $M_D(\CC)$.\footnote{This follows from the fact that the only subvariety of $\CC^E$ that contains $\RR^E$ is $\CC^E$ itself.}
 Thus our final map $\pi$ will be a real transformation, even though we argue via complex transformations.
 
 First we treat property (1). For $p\neq p'\in P$, the complex linear transformations for which $\pi(p) = \pi(p')$ are in a lower-dimensional subspace of $M_D(\CC)$, since the entries of their matrices satisfy a common linear equation. 
 Taking the union of these subspaces for all pairs of distinct points in $P$, we still get a lower-dimensional subspace.
 
 For the second part of property (1), note that if $(p,\gamma^\CC)\in I(P,\Gamma^\CC)$, then obviously $(\pi(p),\pi(\gamma^\CC))\in I(\pi(P), \pi(\Gamma^\CC))$.
 Furthermore, we can choose $\varphi$ so that if $p\in P$ and $p\not\in\gamma$, then $\pi(p)\not\in\pi(\gamma)$.
 To see this, we note that $W = \psi^{-1}(\pi(p))$ is a linear subspace of dimension $D-2$ in $\CC^D$.
 For fixed points $q\in \gamma^\CC$ and $r\in W$, the $\varphi\in M_D( \CC)$ such that $\varphi(q) = r$ lie in a subspace of codimension $D$. 
 Letting $r$ vary in $W$ gives a subspace of codimension two, and letting $q$ vary in $\gamma^\CC$ gives a subvariety of codimension one.
 In other words, for any $\varphi$ outside this subvariety of $M_D(\CC)$
we have $\varphi(\gamma^\CC)\cap W = \emptyset$,
 which means that $\pi(p)\not\in\pi(\gamma^\CC)$.
 
 For property (2), we first ensure that $\pi(\gamma^\CC)$ is a complex algebraic curve in $\CC^2$.
To prove that this is possible we use the Extension Theorem of \cite[Chapter 3, Theorem 1.3 and Corollary 1.4]{cox:ideals},
 which says that if every defining polynomial of $\gamma^\CC$ 
 has $cx_D^N$ as its leading term in $x_D$, then the projection onto the first $D-1$ coordinates is an algebraic curve.
The $\varphi\in M_D(\CC)$ that do not put the polynomials of $\gamma^\CC$ into this form lie in a lower-dimensional subvariety. 
Repeating the same argument $D-3$ times for the projection that removes the last coordinate, each time further restricting the choice of $\varphi$, 
we prove the claim for the projection onto the first two coordinates.
The fact that the degree of $\pi(\gamma^\CC)$ is at most the degree of $\gamma^\CC$ follows from \cite[Lemma 2]{heintz:definability}. 
 
 By the same argument as we used for property (1), we can also ensure that distinct curves are mapped to distinct curves, thereby finishing property (2).
 Indeed, given distinct $\gamma^\CC,\gamma'^\CC$, we can choose any point $p\in \gamma^\CC\backslash \gamma'^\CC$, and ensure that  $\pi(p)\not\in \gamma'^\CC$, which implies that $\pi(\gamma^\CC)$ and $\pi(\gamma'^\CC)$ are distinct; doing this for all pairs of curves from $\Gamma^\CC$ ensures that no two image curves coincide.
 
We have established that there is a $\pi:\RR^D\to\RR^2$ satisfying properties (1) and (2). 
 Let $\Gamma' = \{\gamma^\CC\cap \RR^2: \gamma\in \Gamma\}$, which is a set of distinct real algebraic curves of degree at most $4d$ (as observed in Subsection \ref{subsec:alggeom}). 
 We have established bijections between $I(\pi(P),\Gamma')$, $I(\pi(P), \pi(\gamma^\CC))$, $I(P,\Gamma^\CC)$, and $I(P,\Gamma)$,
 which implies that $\pi(P)$ and $\Gamma'$ have two degrees of freedom.
 Therefore, we can apply Theorem \ref{thm:pachsharir} to them, which finishes the proof.
\end{proof}


\subsection{Symmetries of curves}\label{subsec:symmetries}
We need two elementary results about linear transformations that fix plane algebraic curves.
Given a set $S\subset \mathbb{R}^2$ and a transformation $T:\mathbb{R}^2\to \mathbb{R}^2$,
we say that $T$ {\it fixes} $S$ if $T(S)= S$.
We say that a transformation $T$ is a {\it symmetry} of a plane algebraic curve $C$
if $T$ is an isometry of $\mathbb{R}^2$ and fixes $C$.
Recall that an isometry of $\mathbb{R}^2$ is either 
a rotation, a translation, or a glide reflection 
(a reflection combined with a translation).

In Section \ref{sec:infinite}, we will use the following bound 
on the number of symmetries of a plane algebraic curve.

\begin{lemma}\label{lem:syms}
 An irreducible plane algebraic curve of degree $d$ 
 has at most $4d$ symmetries, unless it is a line or a circle.
\end{lemma}
\begin{proof}
 Suppose the curve $C$ has a translation symmetry $T$. 
 Let $l$ be a line in the direction of $T$ that contains some point $p$ of $C$.
 Then $l\cap C$ must contain the entire trajectory under $T$ of $p$,
 which consists of infinitely many points on $l$.
 By Theorem \ref{thm:bezout}, this implies that $C$ equals $l$.
      
 Suppose $C$ has two rotation symmetries $R_a, R_b$ with distinct centers $a,b$
 and rotation angles $\alpha,\beta$.   
 We claim that then $C$ must also have a translation symmetry, hence equals a line.
 Indeed, consider the composition $R_b\circ R_a$, and
 note that a composition of two rotations is either a translation or a rotation.
 If $R_b\circ R_a$ is a translation, then we are done;
 otherwise it is a rotation $R_c$ with a center $c$ distinct from $a$ and $b$,
 and with angle $\alpha+\beta$.
 Similarly, $R_b^{-1}\circ R_a^{-1}$ is a rotation around a distinct center 
 with angle $-\alpha-\beta$.
 It follows that $R_b^{-1}\circ R_a^{-1}\circ R_b\circ R_a$ is a translation,
 because it cannot be a rotation.
 Indeed, it would have angle $0$, so equal the identity, 
 but it is easily checked that, for instance, it does not fix $a$.
   
 Hence, if $C$ has a rotation symmetry with center $c$,
 then every other rotation symmetry has the same center $c$.
 Let $p$ be any point on $C$ that is not $c$. 
 The image of $p$ under any rotation symmetry then lies on a circle around $c$,
 and no two rotation symmetries give the same point.
 By Theorem \ref{thm:bezout}, either $C$ equals this circle, 
 or it intersects it in at most $2d$ points. 
 Therefore, $C$ is a circle or has at most $2d$ rotation symmetries.
 
 If $C$ had two reflection symmetries with parallel axes of symmetry,
 then $C$ would have a translation symmetry, hence would be a line.
 If $C$ has two reflection symmetries with axes intersecting in $c$,
 then it has a rotation symmetry around $c$, 
 so by the above all axes of reflection symmetries must intersect in the same point.
 Suppose $C$ has $k$ such reflection symmetries. 
 Pick one of them and combine it with each of the $k$ reflections; 
 this gives $k$ distinct rotation symmetries, including the identity.
 This proves that $k\leq 2d$. 
   
 Finally, suppose $C$ has a glide reflection symmetry $G$ which is not a reflection.
 Then $G\circ G$ is a nontrivial translation, so $C$ must be a line.
   
 Altogether, if $C$ is not a line or a circle,
 then it has at most $2d$ rotation symmetries and at most $2d$ reflection symmetries.
\end{proof}

Next we consider the same question for affine transformations (linear transformations combined with translations),
but only in the case where the curve is a conic.

  \begin{lemma}\label{lem:affine}
    Let $T$ be an affine transformation that fixes a conic $C$.
    Then, up to a rotation or a translation, the only possibilities are the following:
    \begin{enumerate}
     \item[(1)] ~$C$ is a hyperbola of the form $y^2 + sxy=t$, with $s,t\neq 0$, and for some real $r\neq 0$
     \begin{align*}
     T(x,y) &= \left(rx+\frac{r^2-1}{rs} y, ~~\frac{1}{r}y\right)
     ~~\text{or}\\
     T(x,y) &= \left(-rx-\frac{r^2-1}{rs} y, ~~rsx + ry\right);
     \end{align*}
     \item[(2)] ~$C$ is an ellipse of the form $s^2x^2 + t^2y^2 = 1$, with $s,t\neq 0$,  and 
     for some $\theta\in [0,2\pi)$
     $$\hspace{-5pt}T(x,y) = \left((\cos\theta)x \pm \frac{t}{s} (\sin\theta)y, 
                      \frac{s}{t}(\sin\theta)x \mp(\cos\theta) y\right);$$
     \item[(3)] ~$C$ is a parabola of the form $y = s x^2$, with $s\neq 0$, and for some $c\in \mathbb{R}$
     $$T(x,y) = (\pm x+c,\pm 2sc x + y + sc^2).$$
    \end{enumerate}
   \end{lemma}
   \begin{proof}
    Suppose $C$ is a hyperbola.
    After a rotation or a translation we can assume that one of the asymptotes is the $x$-axis,
    and the other asymptote goes through the origin,
    so the hyperbola is of the form $y^2 + sxy=t$.
    Applying the shear transformation $T_1(x,y) = (sx+y, y)$ turns it into
    a hyperbola of the form $xy=t$. 
    Suppose $T_2(x,y) = (ax+by+c, dx+ey+f)$ fixes $xy = t$. 
    Then the equation of the image, 
    $t = (ax+by+c)(dx+ey+f)$,
    should be the same equation 
    (or a scalar multiple, but the constant term excludes that).
    This gives six equations, which one can solve
    to get either $T_2(x,y) = (rx, y/r)$, or $T_2(x,y) = (y/r,rx)$. 
    Then it follows that the only affine transformations fixing the original hyperbola
    are of the form $T_1^{-1}\circ T_2\circ T_1$,
    which gives the two forms in the lemma.

    We leave it to the reader to check the other two cases in detail.
    For $C$ an ellipse, we similarly apply a rotation to put it in the given form,
    then apply an expansion $T_1(x,y) = (sx,ty)$ to make it a circle.
    Then we check that rotations around the origin,
    possibly combined with a reflection in a line through the origin, 
    are the only affine transformations 
    that fix a circle around the origin. 
    For $C$ a parabola, a rotation puts it in the given form, 
    and then one can check directly from the equations 
    that the two given forms are the only ones.
   \end{proof}


\section{Proof of Theorem 1.1 and 1.2.}\label{sec:proofoftheorems}
We focus on the proof of Theorem \ref{thm:twocurves}, since Theorem \ref{thm:onecurve} will follow directly by noting that the proof of Theorem \ref{thm:twocurves} allows the curves $C_1$ and $C_2$ to be identical.
Note that, as observed in Subsection \ref{subsec:alggeom}, by our definition these curves may be finite sets. 
In that case, the size of this set is bounded in terms of $d$ by Theorem \ref{thm:milnorthom}, and then our Theorems follow trivially by choosing $c_d$ and $c_d'$ large enough.
Thus, we can assume throughout that $C_1$ and $C_2$ are infinite.

First we define suitable sets of points and curves, and we prove several lemmas about them 
(one of which, Lemma \ref{lem:infinite}, is more involved, and we defer its proof to the next section).
Together these lemmas will enable us to conclude 
that the points and curves essentially form a system with two degrees of freedom, 
so that the incidence bound from Corollary \ref{cor:incidences} can be applied to them. 
This leads to an upper bound on the number of certain quadruples of points. 
On the other hand, a standard argument due to Elekes gives a lower bound on the same quantity, 
inversely proportional to the number of distinct distances. 
Comparing these two bounds at the end of the section,
we obtain the lower bounds on the number of distinct distances
stated in Theorems \ref{thm:onecurve} and \ref{thm:twocurves}. 
 
We have irreducible plane algebraic curves $C_1$ and $C_2$ of degree at most $d$,
given by polynomial equations (of minimum degree)
$$C_1:f_1(x,y)=0,~~~~~~C_2:f_2(x,y)=0.$$
We also have sets $S_1$ on $C_1$ and $S_2$ on $C_2$ with $|S_1| = m$ and $|S_2| = n$;
we write $S_1 = \{p_1,\ldots,p_m\}$ and $S_2 = \{q_1,\ldots,q_n\}$.
We allow $C_1$ and $C_2$ to be the same curve, 
a possibility that will be crucial to the proof of Theorem \ref{thm:onecurve}.
We make the following assumptions, 
which will be justified later.
\begin{assumption}\label{assumption}
 We assume that the following hold:\\
 \emph{(1)} Neither $C_1$ nor $C_2$ is a vertical line;\\
 \emph{(2)} The sets $S_1$ and $S_2$ are disjoint;\\
 \emph{(3)} If $C_1$ (resp. $C_2$) is a circle, then its center is not in $S_2$ (resp. $S_1$);\\
 \emph{(4)} If $C_1$ (resp. $C_2$) is a circle, 
 any concentric circle contains at most one point of $S_2$ (resp. $S_1$);\\
 \emph{(5)} If $C_1$ (resp. $C_2$) is a line, 
 there is at most one point of $S_2$ (resp. $S_1$) 
 on any union of a line parallel to $C_1$ (resp. $C_2$)
 with its reflection in $C_1$ (resp. $C_2$);\\
 \emph{(6)} If $C_1$ (resp. $C_2$) is a line, 
 any orthogonal line contains at most one point of $S_2$ (resp. $S_1$). 
\end{assumption}

We will define a new curve $C_{ij}$ in $\RR^4$ 
for each pair of points $p_i, p_j\in S_1$, written as
$$p_i = (a_i, b_i),~~~~p_j = (a_j, b_j).$$
Let $q$ and $q'$ be points on $C_2$ (not necessarily in $S_2$), written as
$$q = (x, y),~~~~q' = (x', y').$$
We think of $q$ and $q'$ as varying along $C_2$, 
while $p_i$ and $p_j$ are kept fixed on $C_1$.
For $1\leq i,j \leq m$, we define $C_{ij}$ to be the algebraic curve in $\RR^4$ 
consisting of all points $(q,q') = (x,y,x',y')$ satisfying
\begin{align}\label{eq:defcij}
\begin{split}
 f_2(x,y)=0,&~~~f_2(x',y')=0,\\
 (x-a_i)^2 + (y-b_i)^2 &= (x'-a_j)^2 + (y'-b_j)^2.
 \end{split}
\end{align}
In Lemma \ref{lem:intersect} we will prove that $C_{ij}$ has complex dimension one,
which implies that it is indeed a real algebraic curve (by our definition, see Subsection \ref{subsec:alggeom}).
   
Let $P$ be the set of points $(q_s,q_t) =(x_s,y_s,x_t,y_t)\in \RR^4$ 
for any $q_s,q_t\in S_2$.
Note that $(q_s,q_t) \in P$ lies on $C_{ij}$ if and only if
$$d(p_i, q_s) = d(p_j, q_t),$$
so a point on $C_{ij}$ corresponds to points $q_s,q_t\in S_2$
that are equidistant from $p_i$ and $p_j$, respectively.
Therefore, an incidence of $C_{ij}$ with $P$ 
corresponds to a quadruple $(p_i,p_j,q_s,q_t)$
such that $d(p_i, q_s) = d(p_j, q_t)$.
These are the quadruples that we will find upper and lower bounds for.
 
We let $\Gamma$ be the set of curves $C_{ij}$ for $1\leq i,j \leq m$.
Some pairs of these curves may coincide as sets of points,
but we will consider them as different curves, so $|\Gamma| =m^2$. 
We would like $P$ and $\Gamma$ to form a system with two degrees of freedom,
but this is false if some pairs of curves have a common component,
which would mean they have infinite intersection.
This can in fact occur for certain curves $C_1$ and $C_2$,
even if they are not lines or circles.
To overcome this obstacle, we will analyze when exactly the curves $C_{ij}$ can have infinite intersection,
which leads to the following lemma, stating that this obstacle is relatively rare.
We will defer the longer proof of this lemma to the next section and the appendix, 
and first use it to complete the proof of Theorems \ref{thm:twocurves} and \ref{thm:onecurve}.
 
 
\begin{lemma}\label{lem:infinite}
 If $C_1$ and $C_2$ are not parallel lines, orthogonal lines, or concentric circles,
 then there is a subset $\Gamma_0\subset \Gamma$ with $|\Gamma_0|\leq 4dm$ such that
 no three curves in $\Gamma\backslash \Gamma_0$ have infinite intersection.
\end{lemma}
 
Next we show that when two curves have finite intersection, 
the number of their intersection points is bounded in terms of $d$.
This essentially follows from B\'ezout's inequality in $\CC^4$, 
but we will deduce it from the bound in Theorem \ref{thm:milnorthom} on the number of connected components of a real zero set.

In the proof we make use of the fact that the curves $C_{ij}$ have two defining equations in common,
or in other words, they lie on a common surface.
We define $S$ to be this surface, i.e. the set of $(x,y,x',y')\in\RR^4$ 
for which $f_2(x,y)=0$ and $f_2(x',y') = 0$.
It is in fact the Cartesian product of two copies of $C_2$,
which implies that it does  indeed have complex dimension two.


\begin{lemma}\label{lem:intersect}
 Each curve $C_{ij}$ has complex dimension one.
 If $|C_{ij}\cap C_{kl}|$ is finite, 
 then $|C_{ij}\cap C_{kl}| \leq 16 d^4$.
 For any curve $C_{ij}\in \Gamma\backslash \Gamma_0$, 
 there are at most $2d^2$ curves $C_{kl}\in \Gamma\backslash \Gamma_0$
 such that $|C_{ij}\cap C_{kl}|$ is infinite. 
\end{lemma}
\begin{proof}
 Let $C_{ij}^{\CC}$ and $S^{\CC}$ be the complex varieties corresponding to $C_{ij}$ and $S$.
 Note that $C_{ij}^\CC =S^{\CC}\cap Z_{\CC}(F)$
 for $F = (x-a_i)^2 + (y-b_i)^2 - (x'-a_j)^2 - (y'-b_j)^2$.
 To prove that $C_{ij}^{\CC}$ has dimension one, 
 we use the following fact (see \cite[Exercise 11.6]{harris:algebraic}): 
 If $X$ is an irreducible variety in $\CC^n$, 
 and $F$ is any polynomial in $\CC[x_1,\ldots, x_n]$ 
 that does not vanish on all of $X$, 
 then $\dim(X\cap Z_{\CC}(F)) =\dim(X) -1$.
 We observe that $S^{\CC}$ is two-dimensional and irreducible because 
 it is a product of two one-dimensional irreducible varieties
 (see \cite[Exercise 5.9]{harris:algebraic} and the remark before \cite[Theorem 11.12]{harris:algebraic}).
 Then all we have to show is that $F$ does not vanish on all of $S^{\CC}$.
 But if it did, then every point $q_s$ would be at the same distance from $p_i$,
 which is not the case by Assumption \ref{assumption}.3.
 This proves the first claim of the lemma.

 The intersection points $(x,y,x',y')\in C_{ij}\cap C_{kl}$ satisfy the four equations
 \begin{align*}
  (x-a_i)^2 + (y-b_i)^2 &= (x'-a_j)^2 + (y'-b_j)^2,\\
  (x-a_k)^2 + (y-b_k)^2 &= (x'-a_l)^2 + (y'-b_l)^2,\\
  f_2(x,y)=0,~~~~~~&~~~~~~~~~f_2(x',y')=0,
 \end{align*}
each of which has degree at most $d$. 
 By Theorem \ref{thm:milnorthom}, it follows that $C_{ij}\cap C_{kl}$ has
 at most $(2d)^4$ connected components.
 If this intersection is finite, every point is a connected component,
 so the number of points is at most $16d^4$, proving the second claim.
 
 For the last claim, observe that if $|C_{ij}\cap C_{kl}|$ is infinite, 
 then so is $|C_{ij}^\CC\cap C_{kl}^\CC|$,
 which implies that $C_{ij}^\CC$ and $C_{kl}^\CC$ have a common component.
 No three curves $C_{ij}\in\Gamma\backslash \Gamma_0$ 
 have infinite real intersection by Lemma \ref{lem:infinite},
 so no three of the corresponding $C_{ij}^\CC$ share a component that contains infinitely many real points.
 Fix a curve $C_{ij}\in\Gamma\backslash \Gamma_0$.
 Then an irreducible component of $C_{ij}^\CC$ that has infinitely many real points 
 is shared with at most one other $C_{kl}^\CC$.
The curve $C_{ij}^\CC$ has degree at most $2d^2$ and thus at most $2d^2$ irreducible components.
 It follows that at most $d^2$ curves $C_{kl}^\CC$ share with $C_{ij}^\CC$
 a component that contains infinitely many real points,
 which implies that there are at most $2d^2$ curves $C_{kl}$ with which
 $C_{ij}$ has infinite real intersection.
\end{proof}

The two lemmas above let us conclude that, 
although $P$ and $\Gamma$ need not have two degrees of freedom, 
we can partition them into subsets that do.
For each of these subsets we can then bound the number of incidences.
  
  
\begin{lemma}\label{lem:partition}
 Let $L=2d^2+1$.
 There are partitions of $P$ into $P_0,\ldots,P_L$ and 
 $\Gamma$ into $\Gamma_0,\ldots,\Gamma_L$
 such that $|\Gamma_0|\leq 4dm$ and $|P_0| \leq 4dn$,
 and such that for all $1\leq \alpha, \beta \leq L$, 
 the pair $P_\alpha, \Gamma_\beta$ forms a system with two degrees of freedom,
 with multiplicity $M = 16d^4$.
\end{lemma}
\begin{proof}
 Let $\Gamma_0$ be the subset given by Lemma \ref{lem:infinite}, 
 so $|\Gamma_0|\leq 4dm$. 
 We define a graph $G$ with vertex set $\Gamma\backslash \Gamma_0$, 
 connecting two vertices by an edge if the corresponding curves have infinite intersection.
 By Lemma \ref{lem:intersect}, 
 a curve in $\Gamma\backslash \Gamma_0$ has infinite intersection with at most $2d^2 = L-1$ other curves, 
 so the graph has maximum degree $L-1$. 
 It follows that the chromatic number of $G$ is bounded by $L$,
 which means that we can partition the vertices into $L$ independent sets.
 In other words, we can partition $\Gamma\backslash \Gamma_0$ into 
 $L$ subsets $\Gamma_1,\ldots,\Gamma_L$ 
 so that no two curves in the same $\Gamma_\beta$ have infinite intersection.
 Lemma \ref{lem:intersect} then implies that 
 they intersect in at most $16d^4$ points.
 
 To show that a bounded number of curves passes through two points,
 we can reverse the roles of $C_1$ and $C_2$.
 We let $\widetilde{C}_{st}$ be the resulting curves in $\RR^4$, 
 defined analogously to equation \eqref{eq:defcij}.
 So, given $(x_s,y_s),(x_t,y_t)\in C_2$, $\widetilde{C}_{st}$ is 
 the set of all points $(a_i,b_i, a_j,b_j)$ satisfying
 \begin{align*}
 f_1(a_i,b_i)=0,&~~~f_1(a_j,b_j)=0,\\
 (x_s-a_i)^2 + (y_s-b_i)^2 &= (x_t-a_j)^2 + (y_t-b_j)^2.
 \end{align*}
 By the statement analogous to Lemma \ref{lem:infinite},
 there is a subset $\widetilde{\Gamma}_0$ of $4dn$ of these curves $\widetilde{C}_{st}$ 
 such that in the remainder no three curves have infinite intersection.
 Let $P_0$ be the set of points $(q_s,q_t)\in \RR^4$ 
 corresponding to the curves $\widetilde{C}_{st}$ in $\widetilde{\Gamma}_0$.
 
 We define a graph $H$ with vertex set $P\backslash P_0$,
 connecting two points $(q_s,q_t),(q_{s'},q_{t'})$ 
 if the corresponding curves $\widetilde{C}_{st}$ and $\widetilde{C}_{s't'}$
 have infinite intersection.
 As in the case of $G$, we can partition $P\backslash P_0$ into subsets $P_1,\ldots, P_L$
 so that for any two points $(q_s,q_t),(q_{s'},q_{t'})$ 
 in the same $P_\alpha$ with $\alpha \geq 1$,
 the curves $\widetilde{C}_{st}$ and $\widetilde{C}_{s't'}$
 do not have infinite intersection.
 It follows that there are at most $16d^4$ curves from $\Gamma$ 
 passing through any two points from the same $P_\alpha$.
 
 This establishes, for all $\alpha, \beta\geq 1$,  that $P_\alpha$ and $\Gamma_\beta$ 
 form a system with two degrees of freedom, with $M = 16d^4$.
\end{proof}


Applying Corollary \ref{cor:incidences} now gives an incidence bound
for each combination of a point set and a curve set from the partitions.

\begin{lemma}\label{lem:mainincidences}
For all $1\leq \alpha,\beta\leq L$ we have
$$|I(P_\alpha,\Gamma_\beta)| \leq A_d\cdot \max\left\{m^{4/3}n^{4/3}, m^2, n^2\right\}.$$
\end{lemma}


For the relatively few curves and points that we placed in $\Gamma_0$ and $P_0$, we can easily establish the following incidence bound.

\begin{lemma}\label{lem:minorincidences}
 We have 
 $$|I(P,\Gamma_0)| \leq 8d^2mn~~~~~\text{and}~~~~~
 |I(P_0,\Gamma)| \leq 8d^2mn.$$
\end{lemma}
\begin{proof}
 Each $C_{ij}\in\Gamma_0$ has at most $2dn$ incidences with a point $(q_s,q_t)\in P$.
This follows from the fact that there are $n$ choices of $q_s\in S_2$, 
and for each of those, the corresponding $q_t\in S_2$ can be found 
by intersecting $C_2$ with a circle around $p_j$ of radius $d(p_i,q_s)$. 
This gives at most $2d$ solutions by Theorem \ref{thm:bezout}, 
unless $C_2$ equals that circle,
which cannot happen by Assumption \ref{assumption}.3.
Therefore, we have $|I(P,\Gamma_0)| \leq 2dn\cdot 4dm = 8d^2mn$.
The second inequality follows by applying the same argument to the curves $\widetilde{C}_{st}$ 
defined in the proof of Lemma \ref{lem:partition}.
\end{proof}


Before finally proving the main theorems,
 we need the following observation about a certain set of quadruples.
 This observation is a key element in the ``Elekes-Sharir transformation'',
 as introduced in \cite{elekes:incidences} and used in \cite{guth:distances,sharir:distances}.
 Let $Q$ be the set of quadruples $(p_i, p_j, q_s, q_t)$, with $1\leq i,j\leq m$ and $1\leq s,t \leq n$,
 such that $d(p_i,q_s) = d(p_j,q_t)$, 
 and let $D$ be the set of distances between $S_1$ and $S_2$.
 
 
 \begin{lemma}\label{lem:quads}
 We have
   $$|Q| \geq \frac{m^2n^2}{|D|}.$$
 \end{lemma}
 \begin{proof}
  Write $E_d = \{(p,q)\in S_1\times S_2: |pq| = d\}$ for $d\in D$.  
  Then we have
  $$|Q|\geq \sum_{d\in D} |E_d|^2
  \geq \frac{1}{|D|} \left( \sum_{d\in D} |E_d|   \right)^2 = \frac{(mn)^2}{|D|}$$
  by the Cauchy-Schwarz inequality.
 \end{proof}

 
\begin{proof}[Proof of Theorem \ref{thm:twocurves}.]
 First we establish Assumption \ref{assumption}.
 We rotate the coordinate axes so that neither $C_1$ nor $C_2$ is a vertical line.
 We make $S_1$ and $S_2$ disjoint by arbitrarily removing at most half the points of each.
 We remove at most two more points so that if one of $C_1,C_2$ is a circle, then its center is not in the other set.
 For the fourth part of the assumption, if $C_2$ is a circle,
 we observe that since $C_1$ is not a circle concentric with $C_2$, 
 $S_1$ can contain at most $2d$ points of any concentric circle.
 We remove at most $2d-1$ points from $S_1$ from every concentric circle,
 which leaves at least a $1/(2d)$ fraction of the points.
 We do the same for $S_2$.
 In case $C_1$ or $C_2$ is a line,
 we do an analogous removal from every orthogonal line,
 and from every union of a parallel line and its reflection.
 This leaves at least a fraction $1/(2d^2)$ of the points,
 so that the fifth and sixth parts of the assumption are satisfied.
 Altogether these steps leave at least $m/(4d^2)$ points in $S_1$ and $n/(4d^2)$ in $S_2$.
 Now we redefine $S_1$ and $S_2$ to be the point sets resulting from these modifications.

 Combining the bounds from Lemmas \ref{lem:mainincidences} and \ref{lem:minorincidences}, we obtain
 \begin{align*}
  \begin{split}
|I(P,\Gamma)| &\leq |I(P_0, \Gamma)| + |I(P,\Gamma_0)| + \sum_{\alpha,\beta\geq 1} |I(P_\alpha,\Gamma_\beta)| \\
   & \leq 16d^2mn +\sum_{\alpha,\beta\geq 1} A_d\cdot \max\left\{m^{4/3}n^{4/3}, m^2, n^2\right\} \\
   &\leq B_d \cdot \max\left\{ m^{4/3}n^{4/3}, m^2, n^2\right\},
  \end{split}
 \end{align*}
 for the constant $B_d = 6d^4 A_d$, noting that the sum has at most $L^2 \leq 5d^4$ terms.

 On the other hand, by our definitions, an incidence of a curve in $\Gamma$ with a point in $P$
 corresponds exactly to a quadruple $(p_i,p_j, q_s,q_t)$ satisfying $d(p_i, q_s) = d(p_j, q_t)$.
 Combined with Lemma \ref{lem:quads}, this gives
 $$\frac{m^2n^2}{|D|} \leq |Q| = |I(P,\Gamma)|\leq  B_d \cdot\max\left\{ m^{4/3}n^{4/3}, m^2, n^2\right\}.$$
 This implies $|D|\geq c_d' \cdot\min\{m^{2/3}n^{2/3}, m^2, n^2\}$ for the constant $c_d' = 1/(16d^4B_d)$, 
 which also accounts for the points removed at the start of this proof.
\end{proof}


\begin{proof}[Proof of Theorem \ref{thm:onecurve}.]
 We have a curve $C$ of degree $d$, not containing a line or a circle, with a set $S$ of $n$ points on it.
 It has a defining polynomial of degree $d$, which has at most $d$ factors, so the curve has at most $d$ irreducible components. 
 Then there must be a component with at least $n/d$ points; call it $C^*$ and set $S^* = S\cap C^*$.
 Applying the proof above to $C_1 = C^*$, $C_2 = C^*$ shows  
 that $S$ determines at least $c_d n^{4/3}$ distinct distances.
\end{proof}

\noindent \textbf{Remark on the dependence of $c_d$ on $d$.}\\
 With the proof above, the constant $c_d$ in Theorem \ref{thm:onecurve} would come out to be 
 $c_d = cd^{-20/3}$ for some absolute constant $c$.
 Roughly speaking, 
 we get a factor $d^{8/3}$ from the application of Corollary \ref{cor:incidences}
 (using the more precise $C_{\delta,M}=C\cdot \delta^{2/3}\cdot M^{1/3}$ in Theorem \ref{thm:pachsharir}, $M = 16d^4$, and the fact that the projected curves have degree $\delta \leq 2d^2$), 
 and a factor $d^4$ from splitting up $P$ and $\Gamma$ in Lemma \ref{lem:partition}.
 This gives $c_d = c\cdot (d^{8/3}\cdot d^4)^{-1} = c\cdot d^{-20/3}$. 
 For $c_d'$ in Theorem \ref{thm:twocurves}, 
 we would get another factor $d^4$, 
 to account for the removed points in case $C_1$ or $C_2$ is a circle or a line.
 
 To improve the first factor $d^{8/3}$, we could replace Theorem \ref{thm:milnorthom} by 
 a refined bound due to Barone and Basu \cite{barone:bezout}, 
 which takes into account the fact that the defining polynomials may have different degrees.
 This would replace the factor $d^{8/3}$ by $d^2$.
 Furthermore, if we could replace Lemma \ref{lem:infinite} by a similar statement
 for double rather than triple intersections (which we expect to be true), 
 it would make it unnecessary to partition $P$ and $\Gamma$ as in Lemma \ref{lem:partition}, 
 removing the factor $d^4$. 
 Together these two improvements would give $c_d = c\cdot d^{-2}$.
 
 Note that, given an arbitrary set of $n$ points in $\RR^2$,
 one can pass an algebraic curve of degree roughly $\sqrt{n}$ through these points.
 Therefore, a constant $c_d$ on the order of $d^{-2/3}$ would be the best one could hope for,
 because this would imply that $n$ arbitrary points determine $\upOmega(n)$ distances,
 unless many of the points lie on parallel lines or concentric circles.
 
 

\section{Proof of Lemma 3.2.}
\label{sec:infinite}
Our proof of Lemma \ref{lem:infinite} requires four further lemmas that are established in this section.
They will be combined at the very end of the section to deduce Lemma \ref{lem:infinite}.
 
We are going to analyze how two curves $C_{ij}$ and $C_{kl}$ could have infinite intersection.
The most clear-cut case is when $d(p_i, p_k) = d(p_j, p_l)$, 
because then, by Lemma \ref{lem:samedist}, infinite intersection implies the existence of a symmetry of $C_2$. 
This is a real possibility, as will become clear in the proof of Lemma \ref{lem:samedist},
but it cannot happen too often,
because the number of symmetries is bounded if $C_2$ is not a line or circle.
  
On the other hand, when $d(p_i, p_k) \neq d(p_j, p_l)$, 
we expect that $C_{ij}$ and $C_{kl}$ cannot have infinite intersection.
However, we were only able to prove the weaker statement 
that no three curves $C_{ij}$, $C_{kl}$, and $C_{qr}$ can have infinite intersection in this case,
which suffices for our purposes.
We prove this in Lemma \ref{lem:diffdist} when $C_2$ has degree at least 3, and in Lemma \ref{lem:conic} for $C_2$ a conic.
When $C_2$ is a line, we prove a stronger statement in Lemma \ref{lem:line}.

It is worth pointing out where we use that our curves are not parallel lines, orthogonal lines, or concentric circles.
The fact that the curves are not parallel lines allows us to make Assumption \ref{assumption}.5, 
which lets us bound the number of translation symmetries.
Similarly, not having concentric circles allows Assumption \ref{assumption}.4, letting us bound the number of rotation symmetries.
Together, these observations imply that an infinite intersection like in Lemma \ref{lem:samedist} cannot occur too often.
Finally, because our curves are not orthogonal lines, we can make Assumption \ref{assumption}.6, 
which, together with Assumption \ref{assumption}.5, lets us exclude large intersections in Lemma \ref{lem:line}.

\begin{lemma}\label{lem:samedist}
 If $d(p_i, p_k) = d(p_j, p_l)$ and $C_{ij}$ and $C_{kl}$ have infinite intersection,
 then $C_2$ has a symmetry that maps $p_i$ to $p_j$ and $p_k$ to $p_l$.
\end{lemma}
\begin{proof}
 A point $(x,y,x',y') = (q,q')\in C_{ij}\cap C_{kl}$ 
 corresponds to a pair of points $q, q'\in C_2$ 
 such that 
 $$d(p_i, q) = d(p_j, q') ~~~\text{and}~~~d(p_k, q) = d(p_l, q').$$
 It follows that
 \begin{align*}
  \begin{split}
   \{ (d_1,d_2):&~\exists (q, q')\in C_{ij}\cap C_{kl}~\text{such that}~ d_1 = d(p_i, q),d_2 = d(p_k, q)\}\\
   = \{ (d_1,d_2):&~\exists (q, q')\in C_{ij}\cap C_{kl}~\text{such that}~ d_1 = d(p_j, q'),d_2 = d(p_l, q')\}.
  \end{split}
 \end{align*}
 Call this set of pairs of distances $D$.
 Since $C_{ij}$ and $C_{kl}$ have infinite intersection,
 $D$ must be an infinite set.
 
 The idea of the proof is to reconstruct the points of $C_2$ using the distance pairs from $D$,
 by finding the points that respectively have those distances from the points $p_i$ and $p_k$.
 The resulting set of points should consist of an infinite subset of $C_2$, 
 together with its reflection in the line $p_ip_k$.
 The image of this set under the rotation that maps $p_i,p_k$ to $p_j,p_l$ 
 should again have infinite intersection with $C_2$, 
 because $C_2$ should have points at the same distance pairs from $p_j, p_l$.  
 We will see that this implies that $C_2$ has a symmetry.
  
 To make this more precise, 
 let $E$ be the set of all points that arise in this way from a pair of distances in $D$:
 $$E = \{q\in \RR^2:\left (d(p_i,q), d(p_k,q)\right)\in D\}.$$
 Let $M$ be the reflection in the line $p_ip_k$.
 Set $E_1 = E\cap C_2$ and $E_2 = M(E_1)$; 
 because $D$ is infinite, so are $E_1$ and $E_2$.
 Let $T$ be the rotation that maps $p_i,p_k$ to $p_j,p_l$, if it exists; 
 otherwise there is a translation or glide reflection that maps $p_i,p_k$ to $p_j,p_l$,
 and we call that $T$.
 Then $T$ must place an infinite subset of $E$ onto $C_2$; 
 call this subset $E_1^*$, and set $E_2^* = M(E_1^*)$.
 We distinguish two cases:
 \begin{itemize}
  \item[(1)] If $|E_1\cap E_1^*|$ is infinite, 
  then $F_1 = E_1\cap E_1^*$ is an infinite subset of $C_2$
  such that $T(F_1) \subset C_2$;
  \item[(2)] 
  If $|E_1\cap E_1^*|$ is not infinite, 
  then $F_2 = |E_1\cap E_2^*|$ must be infinite.
  Then $F_2 = E_1\cap E_2^*$ is an infinite subset of $C_2$
  such that $(T\circ M)(F_2)\subset C_2$, since $M$ maps $F_2$ into $E_1^*$, 
  which $T$ maps into $C_2$.
 \end{itemize} 
 In each case we use the following observation
 to deduce that $C_2$ has a symmetry:
 If we have an isometry $T$ of the plane and an infinite subset $A$ 
 of an irreducible algebraic curve $C$ such that $T(A)\subset C$,
 then $T(C) = C$, i.e., $T$ is a symmetry of $C$.
 This holds because $T(C)$ is also an irreducible plane algebraic curve, 
 so by Theorem \ref{thm:bezout} it either has finite intersection with $C$, or equals it.
  
 If $T$ is a rotation, then in case (1) $C_2$ has the rotation symmetry $T$,
 while in case (2) it has the reflection symmetry $T\circ M$.
 If $T$ is a translation, then in case (1) it has the translation symmetry $T$,
 while in case (2) it has the glide reflection symmetry $T\circ M$.
\end{proof}
  

\begin{lemma}\label{lem:diffdist}
 Suppose we have $p_i,p_j,p_k,p_l,p_q,p_r\in S_1$ satisfying the following inequalities:
 $d(p_i, p_k) \neq d(p_j, p_l)$, 
 $d(p_i, p_q) \neq d(p_j, p_r)$, and 
 $d(p_k, p_q) \neq d(p_l, p_r)$.
 Then
 $$|C_{ij}\cap C_{kl} \cap C_{qr}| \leq 2d,$$
 unless $C_2$ is a conic or a line.
\end{lemma}
\begin{proof}
 A point in $|C_{ij}\cap C_{kl} \cap C_{qr}|$ corresponds 
 to two points $q_s,q_t\in C_2$
 such that the distances from $p_i,p_k,p_q$ to $q_s$ are equal 
 to those from $p_j, p_l,p_r$ to $q_t$.
 We will show that the set of such points $q_t$ (or $q_t$) 
 is forced to lie on a conic or a line, 
 so by Theorem \ref{thm:bezout} $C_2$ contains at most $2d$ of them, 
 unless $C_2$ is a conic or a line.
 
 We can assume after a rotation that $p_i = (0,0)$ and $p_k = (1,0)$; we think of these points as lying in the $xy$-plane,
 and write $p_q = (a,b)$.
 We think of the points $p_j,p_l$ as lying in a separate $uv$-plane, 
 and there we can assume after a rotation that $p_j = (0,0)$ and $p_l = (L,0)$, with $L\neq 0,1$,
 and $p_r = (c,d)$.
 
 Consider the points $(x,y)$ that have distances $d_1,d_2,d_3$ from respectively $p_i,p_k,p_q$ in the $xy$-plane, 
 and the points $(u,v)$ that have the same distances from $p_j, p_l,p_r$ in the $uv$-plane.
 Then we have the equations
 \begin{align}
  \label{eq:diffdist:first} x^2+y^2 & = u^2 + v^2, \\
  \label{eq:diffdist:second}(x-1)^2 + y^2  & = (u-L)^2 + v^2, \\
  \label{eq:diffdist:third}(x-a)^2 + (y-b)^2 & = (u-c)^2 + (v-d)^2.
 \end{align}
 Subtracting \eqref{eq:diffdist:first} from \eqref{eq:diffdist:second} gives 
 \begin{equation}\label{eq:diffdist:xcoord} 
  x = Lu + \frac{1}{2}(1-L^2).
 \end{equation}
 Subtracting \eqref{eq:diffdist:first} from \eqref{eq:diffdist:third}, 
 and plugging \eqref{eq:diffdist:xcoord} into the result, leads to
 \begin{equation}\label{eq:diffdist:ycoord}
  by = (c - aL) u +dv + \frac{1}{2}(a^2+b^2 -c^2 - d^2 +aL^2 -a).
 \end{equation}

 Plugging the linear equations \eqref{eq:diffdist:xcoord} and \eqref{eq:diffdist:ycoord} 
 into \eqref{eq:diffdist:first} leads to
 \begin{equation}\label{eq:diffdist:quad}
  (b^2L^2 + (c-aL)^2 -b ^2) u^2 + (d^2-b^2) v^2 + 2d(c-aL)uv + l(u,v) = 0
 \end{equation}
 where $l(u,v)$ is a linear function of $u$ and $v$. 
 If this equation is not identically zero,
 this shows that $(u,v)$ must lie on a conic or a line,
 which implies that the original point $q_t$ lies on a conic or line.
 Since $C_2$ is irreducible and not a conic or line,
 it follows that there are at most $2d$ such $q_t$.
 
 This leaves us with the case where \eqref{eq:diffdist:quad} holds identically.
 In that case we can see from the coefficients of the quadratic terms that $b=d=0$ and $c=aL$.
 Plugging these back into \eqref{eq:diffdist:third} easily leads to a contradiction.
 \end{proof}

  
\begin{lemma}\label{lem:conic}
 Suppose we have $p_i,p_j,p_k,p_l,p_q,p_r\in S_1$ satisfying the following inequalities:
 $d(p_i, p_k) \neq d(p_j, p_l)$, 
 $d(p_i, p_q) \neq d(p_j, p_r)$, and 
 $d(p_k, p_q) \neq d(p_l, p_r)$.
 If $C_2$ is a conic then
 $$|C_{ij}\cap C_{kl} \cap C_{qr}| \leq 4.$$
\end{lemma} 
\begin{proof}
 Suppose that $|C_{ij}\cap C_{kl} \cap C_{qr}|\geq 5$.
 Then, similarly to in the previous proof, 
 we have three equations of the form
 \begin{equation}\label{eq:conic:main}
  (x-a_\alpha)^2 + (y-b_\alpha)^2  = (u-c_\alpha)^2 + (v-d_\alpha)^2,
 \end{equation}
 satisfied by at least five pairs of points $(x,y), (u,v)$ on $C_2$. 
 Subtracting the first equation from the second and third gives two linear equations,
 which we can view as an affine transformation $T$ sending $(u,v)$ to $(x,y)$.
 Because $T$ sends five points on $C_2$ to five points on $C_2$,
 it must fix $C_2$, since the image of $C_2$ must be a conic, 
 which could only intersect $C_2$ four times if it was a different conic.   
 Lemma \ref{lem:affine} then tells us which forms $T$ could have.
 We will show that in each case we get a contradiction.
   
 Suppose that $C_2$ is a hyperbola.
 We can apply a rotation to make it of the form $y^2 + s xy = t$
 (note that the rotation moves the points $(a_\alpha,b_\alpha)$ and $(c_\alpha,d_\alpha)$, 
 but does not change the form of the equations, or the condition of the lemma).
 By Lemma \ref{lem:affine}, $T$ must have the form 
 $(u,v) = T(x,y) = (rx+(r^2-1)y/r, y/r)$ 
 (or the second form, which we will leave to the reader). 
 Plugging this into \eqref{eq:conic:main} gives
 $$(x-a_\alpha)^2 + (y-b_\alpha)^2 = \left(rx+\frac{r^2-1}{r}y-c_\alpha\right)^2 
 + \left(\frac{1}{r}y-d_\alpha\right)^2$$
 This equation has a term $x^2$ with coefficient $r^2-1$.
 If $r\neq \pm1$, then this equation describes a different hyperbola than $C_2$,
 so cannot be satisfied by more than four points of $C_2$.   
 If $r=1$, then $T$ is the identity, 
 which would mean that we can put $u = x, v = y$ in \eqref{eq:conic:main}. 
 That would lead to $a_\alpha = c_\alpha,b_\alpha = d_\alpha$ for each $\alpha$, 
 contradicting the assumption of the lemma on the distances between the points.
 Finally, if $r = -1$, we could similarly put $u = -x, v = -y$, 
 leading to $a_\alpha = -c_\alpha,b_\alpha = -d_\alpha$ for each $\alpha$, 
 contradicting the same assumption.
   
 Suppose now that $C_2$ is an ellipse; 
 without loss of generality we can assume that it is of the form $s^2x^2+t^2y^2 =1$.
 By Lemma \ref{lem:affine}, $T$ must have the form
 $(u,v) = T(x,y) = ((\cos\theta)x \pm \frac{t}{s} (\sin\theta)y, 
                     \frac{s}{t}(\sin\theta)x \mp(\cos\theta) y) $.
 Plugging this into \eqref{eq:conic:main} gives
 $$(x-a_\alpha)^2 + (y-b_\alpha)^2 
 = \left((\cos\theta)x \pm \frac{t}{s} (\sin\theta)y-c_\alpha\right)^2 + \left(\frac{s}{t}(\sin\theta)x \mp (\cos\theta) y-d_\alpha\right)^2,$$
 which rearranges to
 $$ \hspace{-100pt} \left(\frac{s^2}{t^2}\sin^2\theta +\cos^2\theta  -1\right)\cdot x^2 
  + \left(\frac{t^2}{s^2}\sin^2\theta + \cos^2\theta  -1\right)\cdot y^2 $$
 $$\hspace{160pt} \pm 2\sin\theta\cos\theta\left(\frac{t}{s} - \frac{s}{t}\right)\cdot xy + l(x,y) = 0.$$
 For this to be an ellipse, the coefficient of $xy$ must be zero,
 so $(t/s - s/t)\sin\theta\cos\theta=0$.
 If $\cos\theta =0$, then the equation takes the form 
 $(s^2/t^2 - 1) x^2 + (t^2/s^2-1) y^2  +l(x,y) = 0$.
 Unless $s=\pm t$ (a case we will consider separately), 
 the $x^2$ and $y^2$ terms have opposite signs, 
 so this cannot be the equation of an ellipse.
 If $\sin\theta = 0$, then $T$ is the identity, 
 which leads to a contradiction as in the hyperbola case.
 It follows that we must have $s= \pm t$.
 This implies that the coefficients of $x^2$ and $y^2$ are also zero, 
 so in fact the polynomial must vanish identically.
 The coefficients of the linear terms then give, after some rearranging, that for each $\alpha$
 $$a_\alpha = (\cos\theta) c_\alpha + (\sin\theta) d_\alpha,~~~~~ 
 b_\alpha = \pm(\sin\theta) c_\alpha \mp (\cos\theta)d_\alpha.$$
 This says exactly that each $(a_\alpha,b_\alpha)$ 
 is the image of $(c_\alpha,d_\alpha)$ under a rotation, 
 or a rotation and a reflection.
 Both are isometries, so the distances between the points are preserved, 
 again contradicting the assumption of the lemma.
   
 Finally, if $C_2$ is parabola $y = cx^2$ and $T(x,y) = (\pm x+c,\pm 2sc x + y + sc^2)$,
 we get 
 $$\hspace{-7pt}(x-a_\alpha)^2 + (y-b_\alpha)^2 = \left(\pm x + c-c_\alpha\right)^2 
 + \left(\pm 2sc x + y + sc^2-d_\alpha\right)^2.$$
 This equation has an $xy$ term with coefficient $\pm4sc$,
 which implies $c=0$, leaving only $T(x,y) = (-x,y)$.
 This is an isometry, which again contradicts the assumption of the lemma.
\end{proof}
   


The remaining case, where $C_2$ is a line, is considerably easier.
The proof is reminiscent of the proofs in \cite{sharir:distances} and \cite{sheffer:few}.

\begin{lemma}\label{lem:line}
 Suppose that $d(p_i, p_k) \neq d(p_j, p_l)$.
 If $C_2$ is a line, then $$|C_{ij}\cap C_{kl}| \leq 4.$$
\end{lemma}
\begin{proof}
We can assume within the proof of this lemma that $C_2$ is the $x$-axis.
Then $C_2\times C_2\subset \RR^4$ is the plane consisting of points of the form $(x,0,x',0)$, 
which we can think of as the $xx'$-plane.
In that plane, $C_{ij}$ is the curve defined by the equation
$$(x-a_i)^2 - (x'-a_j)^2 = b_j^2-b_i^2.$$
Because $d(p_i, p_k) \neq d(p_j, p_l)$, we do not have both $i=j$ and $k=l$,
and by Assumption \ref{assumption}.5, we have $b_j^2 - b_i^2 \neq 0$ when $i\neq j$ and $b_l^2 - b_k^2 \neq 0$ when $k\neq l$.
This implies that at least one of $C_{ij}$ and $C_{kl}$ is a nondegenerate hyperbola.
By Assumption \ref{assumption}.6, we have $a_i\neq a_k$, which implies that $C_{ij}$ and $C_{kl}$ are distinct.
It follows that they intersect in at most four points.
\end{proof}

Finally, we put together the four lemmas in this section to obtain Lemma \ref{lem:infinite}.

\begin{proof}[Proof of Lemma \ref{lem:infinite}.]
 If there is a symmetry $T$ of $C_2$ that maps $p_i$ to $p_j$,
 we will say (just within this proof) that $T$ \textit{respects} $C_{ij}$.
 Suppose that the curves $C_{ij}$ and $C_{kl}$ have infinite intersection and $d(p_i,p_k) = d(p_j, p_l)$.
 By Lemma \ref{lem:samedist}, there is a symmetry $T$ of $C_2$ that respects $C_{ij}$ and $C_{kl}$. 
  
 In case $C_2$ is not a line or a circle, it has at most $4d$ symmetries, by Lemma \ref{lem:syms}.
 Given a fixed symmetry $T$, each $p_i$ is sent to a unique point $p_j$,
 so $T$ respects at most $m$ curves $C_{ij}$.
 Therefore, there are in total at most $4dm$ curves $C_{ij}$ that are respected by some symmetry.
 We let $\Gamma_0$, the set to be excluded, contain all curves $C_{ij}$ 
 that are respected by some symmetry of $C_2$.
 Then $|\Gamma_0|\leq 4dm$.
  
 In case $C_2$ is a line or a circle, it does have many symmetries, 
 but by Assumption \ref{assumption}, 
 there are no $p_i, p_j\in S_1$ such that such a symmetry maps $p_i$ to $p_j$ 
 as in Lemma \ref{lem:samedist},
 so we can take $\Gamma_0$ to be the empty set.
 Indeed, suppose  $C_2$ is a circle and $T$ is a symmetry of $C_2$ with $T(p_i) = p_j$.
 If $p_i\neq p_j$, then they would have to lie on a concentric circle
 (see the proof of Lemma \ref{lem:syms}),
 which is excluded by Assumption \ref{assumption}.4.
 If $p_i= p_j$, then $C_2$ would have to be the circle around $p_i$, 
 contradicting Assumption \ref{assumption}.3.
 A similar argument applies if $C_2$ is a line, using Assumption \ref{assumption}.5.
  
 With $\Gamma_0$ chosen as above,
 it follows that if $C_{ij}, C_{kl}\in \Gamma\backslash \Gamma_0$ have infinite intersection,
 then $d(p_i,p_k) \neq d(p_j, p_l)$.
 Then Lemmas \ref{lem:diffdist}, \ref{lem:conic}, and \ref{lem:line} (together with Assumptions \ref{assumption}.5 and \ref{assumption}.6)
 allow us to conclude that there are no three curves in $\Gamma\backslash \Gamma_0$ that have infinite intersection.

 Applying the argument above to the dual curves $\widetilde{C}_{st}$
 (defined in the proof of Lemma \ref{lem:intersect}) gives the set $P_0$.
 This finishes the proof of Lemma \ref{lem:infinite}.
\end{proof}

  \vspace{30pt} 
  \noindent {\bf Acknowledgements}\\
  The authors would like to thank Filip Mori\'c for some interesting discussions that led to this project, 
  and Natan Rubin, Adam Sheffer, Joshua Zahl, and an anonymous referee for many helpful comments.




\end{document}